\documentclass[12pt]{amsart}

\usepackage{amssymb}

\usepackage{enumerate}

\usepackage{graphicx}

\makeatletter
\@namedef{subjclassname@2010}{%
  \textup{2010} Mathematics Subject Classification}
\makeatother

\newtheorem{thm}{Theorem}

\newtheorem{lem}{Lemma}[section]
\newtheorem{prop}{Proposition}

\theoremstyle{definition}

\numberwithin{equation}{section}

\begin{document}


\baselineskip=17pt



\title[On the asymptotic formula]{On some sums involving the counting function
of non-isomorphic abelian groups}

\author[]{Haihong Fan, Wenguang Zhai}
\address{Department of Mathematics\\
 China University of Mining and Technology\\
Beijing 100083, P. R. China}
\email{fanhaihong1@hotmail.com, zhaiwg@hotmail.com}

\thanks{This work is  supported by the National Natural Science Foundation of China(Grant No. 11971476).}

\date{}

\begin{abstract} Let $a(n)$ denote the number of non-isomorphic abelian groups with $n$ elements. In 1991 Ivi\'{c} proved an asymptotic formula of the sum $\sum_{n\leq x} a(n+ a(n)).$ In this paper, we will prove a sharper  asymptotic formula for this sum.

\end{abstract}

\subjclass[2000]{11N37}

\keywords{finite abelian group,  Dirichlet L-function, Perron's formula}

 \maketitle

\section{\bf Introduction}

Denote by $a(n)$ the number of non-isomorphic abelian groups with $n$ elements. It is well-known that $a(n)$ is $a$ multiplicative function such that $a(p^{\alpha})= P(\alpha)$ for any prime $p$ and integer $\alpha\geq 1$, where $P(\alpha)$ is the number of partitions of $\alpha$.

The properties of the function $a(n)$ is a classical subject in analytic number theory and attracted interests of many authors. The asymptotic behaviour of the sum $\sum_{n\leq x}a(n)$
is an old problem, which goes back to Erd\"{o}s and Szekeres \cite{ES}. For later
improvements, see \cite{Schw,Sch,Sch2}. The local densities of $a(n)$ were studied in \cite{Ri,IGT,Iv}. The iteration problem of $a(n)$ can be found in \cite{EI}.

A. Ivi\'{c} \cite{Iv2} studied the asymptotic behaviour of the sum
$$Q(x):=\sum_{n \leq x} a(n+a(n)).$$
He proved that the asymptotic formula
\begin{equation}
Q(x) = C x + O(x^{11/12 + \varepsilon})
\end{equation}
holds, where   $C > 0$ is a positive constant.

In this paper, we shall prove the following

\begin{thm}\label{thm1} For any $\varepsilon> 0,$ we have the asymptotic formula
\begin{equation}
Q(x) = C x + O(x^{3/4+ \varepsilon}),
\end{equation}
where the $O$-constant depends only on $\varepsilon.$
\end{thm}

For each  $k\geq 2$ fixed, let $d_k(n)$ denote the number of ways $n$ can be written as a product of $k$ natural numbers. It is a classical problem in analytic number
theory to study the counting function $D_k(x):=\sum_{n\leq x}d_k(n).$   We have the
the asymptotic formula
\begin{eqnarray}
D_k(x)=xP_k(\log x)+O(x^{\alpha_k+\varepsilon}),
\end{eqnarray}
where $P_k(t)$ is a polynomial in $t$ of degree $k-1,$ and $0<\alpha_k<1$ is a
real constant(see Ivi\'c \cite{IV} for more details). For example, we have
$\alpha_2\leq 1/2, \alpha_3\leq 1/2,\alpha_4\leq 1/2. $

By the same approach we can prove the following Theorem 2.

\begin{thm}\label{thm2} Let $k\in \{2,3,4\}.$ Then we have
\begin{equation}
\sum_{n\leq x} d_k(n + a(n)) = xQ_k(\log x) + O(x^{3/4+ \varepsilon}),
\end{equation}
where $Q_k(t)$ is a polynomial in $t$ of degree $k-1$ and the $O$-constant depends only on $\varepsilon.$
\end{thm}

Since the proof of Theorem 2 is similar to that of Theorem 1. So for simplicity we only prove Theorem 1. In Section 2 we quote some lemmas which are needed for our proof. In Section 3 we study a sum of $a(n)$ in arithmetic progression. The proof of Theorem 1 will be given in Section 4.

{\bf Notations.} Throughout this paper,   ${\Bbb N}$ and ${\Bbb C}$  denote the set of positive integers and the set of complex numbers, respectively.
We always use $q$ denote  square-free numbers and $s$ denote  square-full numbers. Further, $\varphi$ is Euler's totient function, $\mu$ is the M\"obius function, $\chi$ denotes a Dirichlet character modulo $r$, $L(z,\chi)$ denotes the Dirichlet $L$-function corresponding to $\chi$. For $k\geq 2,$ let $d_k(n)$ denote the number of ways $n$ can be written as a product of $k$ natural numbers, and $d(n)=d_2(n)$. In this paper, $\varepsilon$ always
denotes a small enough positive constant.

\section{\bf Preliminary lemmas}

In order to get the result of the theorem, we shall make use of the following lemmas.

\begin{lem}\label{Lemma_2.1}  Suppose $g(n)\in {\Bbb C}(n\geq 1)$ such that the Dirichlet series
$G(s):=\sum_{n\geq 1}g(n)n^{-s}$ is absolutely convergent for $\sigma>\sigma_a$ and
$$\sum_{n\geq 1}|g(n)|n^{-\sigma}\leq B(\sigma), \ \ \sigma>\sigma_a.$$
Suppose further that $|g(n)|\leq H(n)(n\geq 1),$ where $H(u)>0$ is a function defined
 on $[1,\infty)$ which satisfies $H(u)\asymp H(v)$ for
$u\asymp v.$ Suppose $b>\sigma_a,$
$T\geq 1,$ $x\geq 1,$  $x\notin {\Bbb N}.$  Then we have
\begin{eqnarray*}
\sum_{n\leq x}g(n)&&=\frac{1}{2\pi i}\int_{b-iT}^{b+iT}G(s)\frac{x^s}{s}ds+O\left(\frac{x^b B(b)}{T}\right)\\
&&\ \ +O\left(xH(2x)\min\left(1,\frac{\log x}{T}\right)\right)+O\left(xH(N)\min\left(1,\frac{x}{T\Vert x\Vert}\right)\right),
\end{eqnarray*}
 where $N$ is the integer nearest to $x$ and $\Vert x\Vert=|x-N|.$
\end{lem}

\begin{proof} This is the well-known Perron's formula.
See for example, Theorem 6.5.2 of Pan and Pan \cite{PP} or the formula (A.10) in Appendix of Ivi\'c \cite{IV}.
\end{proof}

\begin{lem} \label{Lemma_2.2}   Let $\chi$ be a Dirichlet character modulo $r$. Then we have
$$L(\sigma+it,\chi) \ll \log r(|t|+ 2)\ \ (\sigma\geq 1).$$
\end{lem}

\begin{proof} See, for example, \cite{PP}.
\end{proof}

\begin{lem} \label{Lemma_2.3} Let $\chi$ be a Dirichlet
character modulo $r$.  Then we have
$$L(\sigma+it,\chi) \ll d(r)  (r(|t|+2))^{\frac{1- \sigma}{2}}
 \log r(|t|+2),\ \ \ (1/2\leq \sigma\leq 1).$$
\end{lem}

\begin{proof} If $\chi$ is a primitive Dirichlet character, then we have
$$L(\sigma+it,\chi) \ll  (r(|t|+2))^{\frac{1- \sigma}{2}}
 \log r(|t|+2),\ \ \ (1/2\leq \sigma\leq 1).$$
 If $\chi$ is not a primitive Dirichlet character, then
 Lemma 2.3 follows from the relation(See, for example, \cite{PP})
 $$L(z,\chi)=L(z,\chi^*)\prod_{p|r}\left(1-\frac{\chi^*(p)}{p^z}\right),\ \ (\Re z>1).$$
\end{proof}

\begin{lem} \label{Lemma_2.4} Let $T\geq 2.$ Then we have
$$\sum_{\chi ( mod \ r)} \int_1^{T} |L(1/2+it,\chi)| ^{2} dt \ll \varphi(r) T \log rT.$$
\end{lem}

\begin{proof} See, for example, \cite{Ra}.
\end{proof}

\begin{lem} \label{Lemma_2.5} Let $T\geq 3$ be a large real number. Suppose $z=\sigma+it, |t|\leq T, |\sigma-1/2|\leq 1/\log T.$ Then we have
$$\sum_{\chi\not= \chi_0}|L(z,\chi)|^2\ll \varphi(r)|z|\log^2 r(|z|+2).$$
\end{lem}
\begin{proof} See, for example, \cite{PP}.
\end{proof}

\begin{lem} \label{Lemma_2.6}  We have
$$  \limsup_{n\rightarrow \infty}\  \log a(n)\cdot \frac{\log\log n}{\log n}=
\frac{\log 5}{4}.$$
\end{lem}

\begin{proof} See Kr\"{a}tzel \cite{Kr}.
\end{proof}

\section{\bf A sum of $a(n)$ in arithmetic progression}

In this section, we will study a sum  of $a(m)$ in arithmetic progression.
Let $x\geq 3$ be a large parameter, $r$ and $k$ are two natural numbers
such that $2\leq r\ll x^{1/2}$ and $1\leq k\ll x^{1/2}$. Define
 \begin{equation}
 T(x; k, r) :=  \sum_{m\leq x,\, m\equiv k\,(mod\,r)} a(m)
 \end{equation}
 which plays an important role in the proof of our theorem.
This sum has been studied by several authors. For example,
  Richert \cite{Ri},  Duttlinger \cite{Du} and Ivi\'c \cite{Ra}.
 Ivi\'c stated
 \begin{equation}
T(x; k, r) = B(r,k) x +O((rx)^{1/2+\varepsilon}), \  B(r,k)= O(1/r).
\end{equation}
Ivi\'c gave a detailed proof of (3.2) for the case $(r,k)=1.$

In this section we give a sharper asymptotic formula than (3.2). We shall prove the following   proposition.

\begin{prop} \label{Prop_1}  Uniformly for $r\ll x^{1/2}$
  we have
\begin{equation}
 T(x; k, r) = \frac{c(r, k)}{r} x + O(d(r)x^{1/2}\log^{2.5} x),
 \end{equation}
where $c(r,k)$ is defined by (3.18) such that $c(r,k)\ll d(r)$ and the $O$-constant is absolute.
\end{prop}


Let $u = (r, k), r = u r_{1}, k = u k_{1}, ( r_{1}, k_{1}) = 1.$
Then $m \equiv k(mod \ r)$ implies that $m = r n + k = u(r_{1} n + k_{1}).$ So we have
\begin{align}
T(x; k, r) & =   \sum_{u (n r_{1}+ k_{1}) \leq x} a( u( n r_{1}+ k_{1}))
= \sum_{\stackrel{n \leq x/u}{ n \equiv k_{1} (mod r_{1})}} a(u n) \\
&=\frac{1}{\varphi(r_{1})} \sum_{\chi \, (mod \,r_{1})} \bar{\chi}(k_{1})
 \left(\sum_{ n\leq x/u} \chi(n) a(u n)\right ),\nonumber
\end{align}
where in the last step we used the orthogonality property of Dirichlet characters.
Note that (3.4) holds for $u=1$ or $r_1=1$.

It suffices for us to evaluate the innermost sum ($X=x/u$)
$$M(X;\chi,u):=\sum_{ n\leq X} \chi(n) a(u n)$$
for any Dirichlet character $\chi$ modulo $r_1.$
Suppose that $ u = \prod_{p} p^{l}.$ Since both
$a(n)$ and $\chi(n)$ are multiplicative, we have by Euler product that ($\Re z>1$)
\begin{align}
D_{u}(z, \chi) & = \sum_{n = 1}^{\infty} \frac{a(u n) \chi (n)}{n^{z}}
  = \prod_{p} \left(\sum_{\alpha = 0}^{\infty} \frac{a( p^{l + \alpha}) \chi( p^{\alpha})}{p^{\alpha z}}\right) \\
& = \prod_{p  \nmid u} \left(\sum_{\alpha = 0}^{\infty} \frac{a( p^{\alpha}) \chi( p^{\alpha})}{p^{\alpha z}}\right) \cdot \prod_{p \mid u} \left(\sum_{\alpha = 0}^{\infty} \frac{a( p^{l + \alpha}) \chi( p^{\alpha})}{p^{\alpha z}}\right)\nonumber\\
& = \prod_{p} \left(\sum_{\alpha = 0}^{\infty} \frac{a( p^{\alpha}) \chi( p^{\alpha})}{p^{\alpha z}}\right) \cdot \prod_{p \mid u} \left(\sum_{\alpha = 0}^{\infty} \frac{a( p^{\alpha}) \chi( p^{\alpha})}{p^{\alpha z}}\right)^{-1}\nonumber\\
&\ \ \ \ \ \ \ \cdot \prod_{p \mid u} \left(\sum_{\alpha = 0}^{\infty} \frac{a( p^{l + \alpha}) \chi( p^{\alpha})}{p^{\alpha z}}\right)\nonumber\\
& = H(z, \chi) F_{u}(z, \chi), \nonumber
\end{align}
where
\begin{align}
& H(z, \chi): = \prod_{p} \left(\sum_{\alpha = 0}^{\infty} \frac{a( p^{\alpha}) \chi( p^{\alpha})}{p^{\alpha z}}\right),\\
& F_{u}(z, \chi): = \prod_{p \mid u} \left(\sum_{\alpha = 0}^{\infty} \frac{a( p^{\alpha}) \chi( p^{\alpha})}{p^{\alpha z}}\right)^{-1} \cdot \prod_{p \mid u} \left(\sum_{\alpha = 0}^{\infty} \frac{a( p^{l + \alpha}) \chi( p^{\alpha})}{p^{\alpha z}}\right).\nonumber
\end{align}

Note that if $u=1,$ then $F_u(z,\chi)\equiv 1.$ If  $u>1,$ it is easy to see that
the Dirichlet series for $F_{u}(z, \chi)$ converges absolutely for $\Re z > 0$ and
we have the estimate
\begin{equation}
F_{u}(z, \chi) \ll \prod_{p\mid u} (1 + a(p^l)p^{- 1/2} )\ll d(u),\ \Re z\geq 1/2,\ \ (u=\prod_{p}p^l).
\end{equation}

By the formula $ 1 + t + 2 t^{2} = (1-t)^{-1}(1 - t^{2})^{-1} (1 + O(t^{3}))$ we have
\begin{equation} \label{3.4}
H(z, \chi) = L(z, \chi) L(2z, \chi^{2}) G(z, \chi),\ (\Re z>1)
\end{equation}
where
\begin{align}
G(z, \chi): &  = \prod_{p} \left(1 - \frac{\chi(p)}{p^{z}}\right)
\left(1 - \frac{\chi(p)^{2}}{p^{2 z}}\right)
 \left(\sum_{\alpha = 0}^{\infty} \frac{a( p^{\alpha}) \chi( p^{\alpha})}{p^{\alpha z}} \right), \nonumber
\end{align}
which is absolutely convergent for $\Re z>1/3.$ So from (3.5),(3.6) and (3.8) we have
\begin{equation}
 D_{u}(z, \chi) = L(z, \chi) L(2z, \chi^{2}) G(z, \chi) F_{u}(z, \chi).
\end{equation}

Let $ 1\leq  T\leq X^{1000}$ be a parameter to be determined. By Lemma 2.1, we have
\begin{align}
M(X;\chi,u) &=\frac{1}{2 \pi i}\int_{1+ \varepsilon - iT}^{1+ \varepsilon + iT} D_{u}(z, \chi) \frac{X^{z}}{z} dz + O\left(\frac{X^{1+ \varepsilon} \log X }{T}\right)\\
&=\frac{1}{2 \pi i}\int_{1+ \varepsilon - iT}^{1+ \varepsilon + iT} L(z, \chi) L(2z, \chi^{2}) G(z, \chi) F_{u}(z, \chi) \frac{X^{z}}{z} dz\nonumber\\
&\ \ \ \ + O\left(\frac{X^{1+ \varepsilon} \log X}{T}\right). \nonumber
\end{align}

Let
\begin{eqnarray*}
&&C_1=\{z=\sigma+Ti:\ 1/2\leq \sigma \leq 1+\varepsilon\},\\
&&C_2=\{z=1/2+ti:\ \frac{1}{\log X}\leq t\leq T\},\\
&&C_3=\{z=\frac{e^{i \theta}}{\log X}:\ -\pi/2\leq \theta\leq \pi/2\},\\
&&C_4=\{z=1/2+ti:\ -T\leq t\leq -\frac{1}{\log X} \},\\
&&C_5=\{z=\sigma-Ti:\ 1/2\leq \sigma \leq 1+\varepsilon\}.
\end{eqnarray*}
By the residue theorem, we have
\begin{equation}
  M(X;\chi,u)=Res_{z=1}L(z, \chi) L(2z, \chi^{2}) G(z, \chi) F_{u}(z, \chi) \frac{X^{z}}{z}
 +\sum_{j=1}^4 \int_j-\int_5 + O \left(\frac{X^{1+ \varepsilon} \log X}{T}\right),
\end{equation}
where
\begin{equation}
\int_j:=\frac{1}{2\pi i}\int_{C_j}L(z, \chi) L(2z, \chi^{2}) G(z, \chi) F_{u}(z, \chi) \frac{X^{z}}{z}dz.
\end{equation}

If $\chi=\chi_0$ is the principal character, then we have
\begin{eqnarray*}
 Res_{z=1}L(z, \chi) L(2z, \chi^{2}) G(z, \chi) F_{u}(z, \chi) \frac{X^{z}}{z}=
\frac{\varphi(r_1)}{r_1}L(2, \chi_0^{2}) G(1, \chi_0) F_{u}(1, \chi_0)X.
\end{eqnarray*}
If $\chi$ is not a principal character, then
\begin{eqnarray*}
 Res_{z=1}L(z, \chi) L(2z, \chi^{2}) G(z, \chi) F_{u}(z, \chi) \frac{x^{z}}{z}=0.
\end{eqnarray*}

If $z=\sigma+Ti\ (1/2\leq \sigma\leq 1),$ then from Lemma 2.2, Lemma 2.3  and (3.7) we have
\begin{equation}
L(z, \chi) L(2z, \chi^{2}) G(z, \chi) F_{u}(z, \chi) \frac{X^{z}}{z}
\ll d(r)d(u)T^{-1}X^\sigma(r_1T)^{\frac{1-\sigma}{2}}\log^2 r_1T.
\end{equation}
If $z=\sigma+Ti\ (1 \leq \sigma\leq 1+\varepsilon),$
then
\begin{equation}
L(z, \chi) L(2z, \chi^{2}) G(z, \chi) F_{u}(z, \chi) \frac{x^{z}}{z}
\ll T^{-1}X^\sigma \log^{2}  r_1T.
\end{equation}
So from (3.13) and (3.14) we have
\begin{equation}
\int_{1}\ll \frac{X^{1+\varepsilon}}{T}+\frac{r_1^{1/4}X^{1/2+\varepsilon}}{T^{3/4}}.
\end{equation}
Similarly,
\begin{equation}
\int_{5}\ll \frac{X^{1+\varepsilon}}{T}+\frac{r_1^{1/4}X^{1/2+\varepsilon}}{T^{3/4}}.
\end{equation}

From (3.4), (3.11), (3,15) and (3.16) we obtain
\begin{eqnarray}
T(x,k,r)&&=\frac{c(r,k)}{r_{1}}X + \frac{1}{\varphi(r_1)}\sum_{\chi(mod\ r_1)}
\overline{\chi}(k_1)\left(\int_2+\int_3+\int_4\right)\\
&&\ \ \ \ \  + O \left(\frac{X^{1+\varepsilon}}{T}+
\frac{r_1^{1/4}X^{1/2+\varepsilon}}{T^{3/4}}\right),\nonumber
\end{eqnarray}
where
\begin{equation}
 c(r,k):= L(2, \chi_0^{2}) G(1, \chi_0) F_{u}(1, \chi_0).
\end{equation}

It is easy to see that
\begin{equation}
 c(r,k)\ll |F_{u}(1, \chi_0)|\ll d(u)\leq d(r).\end{equation}
We have
\begin{eqnarray}
\frac{1}{\varphi(r_1)}\sum_{\chi(mod\ r_1)}
\overline{\chi}(k_1)\left(\int_2+\int_3+\int_4\right)\ll W_1+W_2+W_3,
\end{eqnarray}
where
\begin{eqnarray*}
&&W_1:=X^{1/2}d(u)\log r_1T\times \frac{1}{\varphi(r_1)}\sum_{\chi(mod\ r_1)}\int_{1}^{T}
|L(1/2+it, \chi)    |\frac{dt}{t},\\
&&W_2:=X^{1/2}d(u)\log r_1T\times \frac{1}{\varphi(r_1)}\sum_{\chi(mod\ r_1)}\int_{\frac{1}{\log X}}^{1}|L(1/2+it, \chi)|dt,\\
&&W_3:=X^{1/2}d(u)\log r_1T\times
{\varphi(r_1)}\sum_{\chi(mod\ r_1)}\int_0^{\frac \pi 2}|L(1/2+\frac{e^{i \theta}}{\log X}, \chi)|d\theta.
\end{eqnarray*}

From Lemma 2.4 and Cauchy's inequality we have
\begin{eqnarray*}
 &&\ \ \ \  \frac{1}{\varphi(r_1)}\sum_{\chi(mod\ r_1)}\int_{1}^{T}
|L(1/2+it, \chi)    |\frac{dt}{t}\\
&&\ll \frac{\log T}{\varphi(r_1)}\max_{1\ll T_0\ll T}
\frac{1}{T_0}\sum_{\chi(mod\ r_1)}\int_{T_0}^{2T_0}
|L(1/2+it, \chi)    |dt\nonumber
\\&& \ll \frac{\log T}{\varphi(r_1)}\max_{1\ll T_0\ll T}\frac{1}{T_0}
\left(\sum_{\chi(mod\ r_1)}\int_{T_0}^{2T_0}|L(1/2+it, \chi)    |^2dt\right)^{1/2}
\left(\sum_{\chi(mod\ r_1)}\int_{T_0}^{2T_0}dt\right)^{1/2}\\
&&\ll \frac{\log T}{\varphi(r_1)}\max_{1\ll T_0\ll T}\frac{1}{T_0}
(\varphi(r_1)T_0\log T_0)^{1/2}(\varphi(r_1)T_0 )^{1/2}\\
&&\ll \log^{3/2} T,
\end{eqnarray*}
which implies that
\begin{equation}
W_1\ll X^{1/2}d(u)\log^{5/2} X.
\end{equation}

Similarly to the above argument, from Lemma 2.5 we can get
$$\frac{1}{\varphi(r_1)}\sum_{\chi(mod\ r_1)}\int_{\frac{1}{\log X}}^{1}|L(1/2+it, \chi)|dt\ll \log T, $$
which implies that
\begin{equation}
W_2\ll X^{1/2}d(u)\log^2 X.
\end{equation}
Similarly, we have
\begin{equation}
W_3\ll X^{1/2}d(u)\log^2 X.
\end{equation}

Inserting (3.20)-(3.23) into (3.17) we get
\begin{eqnarray}
T(x,k,r)&&=\frac{c(r,k)}{r_{1}}X +O(X^{1/2}d(u)\log^{2.5} X)\\
&&=\frac{c(r,k)}{r}x+O(x^{1/2}d(r)\log^{2.5} x)\nonumber
\end{eqnarray}
by choosing $T=X^{10}.$
This completes the proof of the proposition.

\section{\bf Proof of  Theorem $1$ }

\label{intro} \setcounter{equation}{0}
\medskip

Each number $n$ can be uniquely written as $n=qs$ such that
$q$ is square-free, $s$ is square-full and $(q,s)=1$.  It is well-known that
  $a(\ell)\equiv 1$ for any square-free $\ell.$ Thus we have
\begin{eqnarray*}
  Q(x)&&=\sum_{n\leq x} a(n + a(n)) = \sum_{\stackrel{q s \leq x}{(q, s) = 1}} a( q s + a(s))\\
&&= \sum_{k \leq A(x)}\ \ \sum_{\stackrel{s \leq x}{a(s)= k}}
 \sum_{\stackrel{q \leq x / s}{(q, s) = 1}} a(q s + k)\nonumber\\
 &&= \sum_{k \leq A(x)}\ \ \sum_{\stackrel{s \leq x}{a(s)= k}}\ \
 \sum_{\stackrel{d^2n \leq x / s}{(d, s) =(n,s)= 1}} \mu(d)a(d^2n s + k)\nonumber\\
 &&=\sum_{k \leq A(x)}\ \ \sum_{\stackrel{s \leq x}{a(s)= k}}\ \
 \sum_{\stackrel{d^2s\leq x}{(d,s)=1}}\mu(d)
 \sum_{\stackrel{n \leq x /d^2 s}{(n,s)= 1}}a(d^2n s + k)\nonumber
\end{eqnarray*}
where $A(x):=\max_{n\leq x}a(n)$ and in the fourth "=" we used  the
 familiar  relation $ \mu^{2} (n) = \sum_{d^{2} \mid n} \mu (d)$.

 Suppose $x^{\varepsilon}\ll y\ll x^{1/2}$ is a parameter to be determined later. We write
 \begin{eqnarray}
   Q(x)= Q_1(x,y)+ Q_2(x,y),
 \end{eqnarray}
 where
\begin{eqnarray*}
 && Q_1(x,y): =\sum_{k \leq A(x)}\ \ \sum_{\stackrel{s \leq x}{a(s)= k}}\ \
 \sum_{\stackrel{d^2s\leq y}{(d,s)=1}}\mu(d)
 \sum_{\stackrel{n \leq x /d^2 s}{(n,s)= 1}}a(d^2n s + k),\\
 && Q_2(x,y): =\sum_{k \leq A(x)}\ \ \sum_{\stackrel{s \leq x}{a(s)= k}}\ \
 \sum_{\stackrel{y<d^2s\leq x}{(d,s)=1}}\mu(d)
 \sum_{\stackrel{n \leq x /d^2 s}{(n,s)= 1}}a(d^2n s + k).
\end{eqnarray*}

If $d^2s\leq y$, it follows that $s\leq y$ and $k\leq A(y).$ So $Q_1(x,y)$ can be rewritten as
\begin{eqnarray}
   Q_1(x,y) =\sum_{k \leq A(y)}\ \ \sum_{\stackrel{s \leq y}{a(s)= k}}\ \
 \sum_{\stackrel{d^2s\leq y}{(d,s)=1}}\mu(d)
 \sum_{\stackrel{n \leq x /d^2 s}{(n,s)= 1}}a(d^2n s + k).
\end{eqnarray}

By the well-known bound $a(n)\ll n^\varepsilon$ we have (note that $d^2s$ is square-full)
\begin{eqnarray}
Q_2(x,y)&&\ll \sum_{k \leq A(x)}\  \sum_{\stackrel{s \leq x}{a(s)= k}}\ \
\sum_{\stackrel{y<d^{2} s \leq x}{(d, s) = 1}} |\mu (d)|\frac{x^{1+\varepsilon}}{d^2s}\\
&&\ll x^{1+\varepsilon}
\sum_{\stackrel{y<d^{2} s \leq x}{(d, s) = 1}}\frac{|\mu (d)|}{d^2s}
\ll x^{1+\varepsilon}
\sum_{ y<  s \leq x }\frac{1}{ s}
\ll  \frac{x^{1+\varepsilon}}{y^{1/2}}\nonumber
\end{eqnarray}
by using partial summation with the help of the familiar bound
\begin{equation}
\sum_{s\leq u}1\ll u^{1/2}.
 \end{equation}

Now we evaluate the sum $Q_1(x,y).$
We first consider  the innermost sum in $Q_1(x,y).$
By the  elementary formula
$ \sum_{d \mid n} \mu (d)  =[1/n]$ we have
\begin{equation}
 \sum_{n \leq x / d^{2} s,\,(n, s) = 1} a(d^{2} n s + k) = \sum_{\delta \mid s}\mu (\delta) \ \ \sum_{ \delta n_{1} \leq x/ d^{2} s} a( d^{2}\, \delta\, n_{1}\, s + k).
\end{equation}

From (4.2) and (4.5) we have
 \begin{eqnarray}
 \ \ \ \ Q_1(x,y)&& =\sum_{k \leq A(y)}\ \ \sum_{\stackrel{s \leq y}{a(s)= k}} \sum_{\stackrel{d^{2} s \leq y}{(d, s) = 1}}\
\mu (d) \sum_{\delta \mid s}\mu (\delta)  \sum_{ n  \leq x/ d^{2}\delta s} a( d^{2}\, \delta\, n\, s + k)\\
&&=\sum_{k \leq A(y)}\ \ \sum_{\stackrel{s \leq y}{a(s)= k}}\sum_{\stackrel{d^{2} s \leq y}{(d, s) = 1}}\
\mu (d) \sum_{\delta \mid s}\mu (\delta)
 \sum_{\stackrel{ n\leq x+k}{n \equiv k( d^{2}\delta s)}} a(n).\nonumber
\end{eqnarray}

By Proposition 1 we have
\begin{eqnarray}
\sum_{\stackrel{ n\leq x+k}{n \equiv k( d^{2}\delta s)}} a(n)
&&=\frac{c(d^{2}\delta s,k)}{d^{2}\delta s}(x+k)+O(x^{1/2}\log^4 x)\\
&&=\frac{c(d^{2}\delta s,k)}{d^{2}\delta s}x+O(x^{1/2}\log^4 x).\nonumber
\end{eqnarray}

Inserting (4.7) into (4.6) we get
\begin{eqnarray}
Q_1(x,y)=xJ_{1}(y)+O(x^{1/2}\log^4 x\times J_{2}(y)),
\end{eqnarray}
where
\begin{eqnarray*}
J_{1}(y)&&:=\sum_{k\leq A(y)}\sum_{\stackrel{s\leq y}{a(s)=k}}
\sum_{\stackrel{d^{2} s \leq y}{(d, s) = 1}}\
\mu (d) \sum_{\delta \mid s}\mu (\delta)\frac{c(d^{2}\delta s,k)}{d^{2}\delta s}, \\
J_{2}(y)&&:=\sum_{k\leq A(y)}\sum_{\stackrel{s\leq y}{a(s)=k}}
\sum_{\stackrel{d^{2} s \leq y}{(d, s) = 1}}\
|\mu (d)| \sum_{\delta \mid s}|\mu (\delta)|.
\end{eqnarray*}
Obviously, we have  $ \sum_{\delta \mid s}|\mu (\delta)|
\leq d(s)\ll s^\varepsilon$. So we have by (4.4) that
\begin{eqnarray}
J_{2}(y)&&\ll
\sum_{\stackrel{d^{2} s \leq y}{(d, s) = 1}}\
|\mu (d)| \sum_{\delta \mid s}|\mu (\delta)|
 \ll y^{\varepsilon}\sum_{s\leq y}1\ll y^{1/2+\varepsilon}.
\end{eqnarray}

 Now we consider $J_{1}(y).$ We can write
 \begin{eqnarray}
\ \ \ \ \ \ J_{1}(y)&&=\sum_{k\leq A(y)}\sum_{\stackrel{s\leq y}{a(s)=k}}\frac 1s
\sum_{\stackrel{d \leq \sqrt{\frac ys}}{(d, s) = 1}}\
\frac{\mu (d)}{d^2} \sum_{\delta \mid s}
\frac{\mu (\delta)c(d^{2}\delta s,k)}{\delta}\\
&& =J_{11}+O(J_{12}),\nonumber
\end{eqnarray}
where
\begin{eqnarray}
J_{11}&&:=\sum_{k\leq A(y)}\sum_{\stackrel{s\leq y}{a(s)=k}}\frac 1s
 \sum_{\stackrel{d=1}{(d, s) = 1}}^\infty\
 \frac{\mu (d)}{d^2} \sum_{\delta \mid s}
\frac{\mu (\delta)c(d^{2}\delta s,k)}{\delta},  \nonumber\\
J_{12}&&:=\sum_{k\leq A(y)}\sum_{\stackrel{s\leq y}{a(s)=k}}\frac{s^\varepsilon}{s}
\sum_{d>\sqrt{\frac ys}}\frac{d^\varepsilon}{d^2},
\nonumber
\end{eqnarray}
where we used the bound $c(r,k)\ll d(r)\ll r^\varepsilon.$

By (4.4) and partial summation it is easy to see that
 \begin{eqnarray}
 J_{12}\ll \sum_{ s\leq y }\frac{s^\varepsilon}{\sqrt s}\frac{1}{\sqrt y}\ll
   y^{-1/2+\varepsilon} .
 \end{eqnarray}

Let  $y_0$ denote the smallest natural number not exceeding $y$ such that
$A(y)=a(y_0).$ Then we can write
\begin{equation}
J_{11}=C+O(E_y),
\end{equation}
 where
\begin{eqnarray}
&&C=\sum_{k=1}^\infty\sum_{\stackrel{s=1}{a(s)=k}}^\infty\frac 1s
 \sum_{\stackrel{d=1}{(d, s) = 1}}^\infty\
  \frac{\mu (d)}{d^2} \sum_{\delta \mid s}
\frac{\mu (\delta)c(d^{2}\delta s,k)}{\delta},\\
&&E_y:=\sum_{k>A(y)} \sum_{\stackrel{s>y_0}{a(s)=k}} \frac 1s
 \sum_{\stackrel{d=1}{(d, s) = 1}}^\infty\
  \frac{|\mu (d)|}{d^2} \sum_{\delta \mid s}
\frac{|\mu (\delta)|c(d^{2}\delta s,k)}{\delta}.\nonumber
\end{eqnarray}

From Lemma 2.6  we get that for any small positive constant $\varepsilon>0,$ the inequality
\begin{eqnarray*}
(1-0.1\varepsilon )\frac{\log 5}{4}\cdot\frac{\log y}{\log\log y}
<\log  A(y)<(1+0.1\varepsilon )\frac{\log 5}{4}\cdot\frac{\log y}{\log\log y}
\end{eqnarray*}
holds for $y_0>y^{1-\varepsilon}.$ If $y_0\leq y^{1-\varepsilon}, $ then
\begin{eqnarray*}
 \log  A(y)&&  <(1+0.1\varepsilon )\frac{\log 5}{4}\frac{\log y^{1-\varepsilon}}{\log\log y^{1-\varepsilon}}\\
 &&=(1-0.9\varepsilon-0.1\varepsilon^2)
 \frac{\log 5}{4}\frac{\log y}{\log\log y+\log(1-\varepsilon)}.
\end{eqnarray*}
The above two formulas imply that
$$(1-0.9\varepsilon-0.1\varepsilon^2)
 \frac{\log 5}{4}\frac{\log y}{\log\log y+\log(1-\varepsilon)}
 <(1-0.1\varepsilon )\frac{\log 5}{4}\cdot\frac{\log y}{\log\log y},
 $$
This is a contradiction if $\varepsilon>0$ is small enough. So we have
$y_0> y^{1-\varepsilon}.$ This fact via (4.4) implies that
\begin{eqnarray}
 E_y&&\ll\sum_{k>A(y)} \sum_{\stackrel{s>y_0}{a(s)=k}} \frac{s^\varepsilon }{s}
 \sum_{\stackrel{d=1}{(d, s) = 1}}^\infty \frac{d^\varepsilon}{d^2}
 \ll
 \sum_{ s>y^{1-\varepsilon} }\frac{s^\varepsilon}{s}\ll y^{-1/2+\varepsilon}.
\end{eqnarray}

 From (4.8)-(4.14) we get
\begin{eqnarray}
Q_1(x,y)=Cx+O(x^{1+\varepsilon}y^{-1/2}+x^{1/2+\varepsilon}y^{1/2}).
\end{eqnarray}

Now  Theorem 1 follows from (4.1), (4.3) and (4.15) by choosing $y=x^{1/2}.$

\subsection*{Acknowledgements}
The author is very grateful to Professor Wenguang Zhai for his enthusiastic guidance. Also the author wants to thank the referee for his helpful and detailed comments.


\end{document}